\documentclass[12pt]{amsart}
\usepackage[dvipsnames,usenames]{color}
\usepackage{hyperref}
\usepackage{graphicx}
\usepackage{epsfig}
\usepackage[latin1]{inputenc}
\usepackage{amsmath}
\usepackage{amsfonts}
\usepackage{amssymb}
\usepackage{amsthm}
\usepackage{amscd}
\usepackage{verbatim}
\usepackage{subfig}
\usepackage{pinlabel}
\usepackage{mathtools}
\usepackage{stmaryrd}
\usepackage{enumerate, enumitem}
\usepackage{array,multirow}

\usepackage[textsize=scriptsize]{todonotes}

\usepackage{tikz}
\usetikzlibrary{cd}
\usetikzlibrary{arrows}
\usetikzlibrary{decorations.pathreplacing}
\usepackage{verbatim}

\newtheorem*{theorem*}{Theorem}
\newtheorem{theorem}{Theorem}[section]
\newtheorem{lemma}[theorem]{Lemma}

\theoremstyle{definition}

\newtheorem{remark}[theorem]{Remark}

\def\Z{\mathbb{Z}}

\title{A note on rational band moves}

\author[D. Chen]{Daren Chen}
\address{Department of Mathematics, Caltech}
\email{darenc@caltech.edu}

\author[J. Hom]{Jennifer Hom}
\address{School of Mathematics, Georgia Institute of Technology}
\email{hom@math.gatech.edu}

\author[M. Kim]{Min Hoon Kim}
\address{Department of Mathematics Education, Kyungpook National University}
\email{minhoonkim@knu.ac.kr}

\author[J. Park]{\\JungHwan Park}
\address{Department of Mathematical Sciences, KAIST}
\email{jungpark0817@kaist.ac.kr}

\author[Z. Wu]{Zhongtao Wu}
\address{Department of Mathematics, The Chinese University of Hong Kong}
\email{ztwu@math.cuhk.edu.hk}

\begin{document}

\maketitle

\begin{abstract} We introduce an oriented rational band move, a generalization of an ordinary oriented band move, and show that 
if a knot $K$ in the three-sphere can be made into the $(n+1)$-component unlink by $n$ oriented rational band moves, then $K$ is rationally slice. 
\end{abstract}

\section{Introduction}
One way to show that a knot $K$ in $S^3$ is \emph{slice} (in fact, \emph{ribbon}) is to find $n$ oriented bands that can be added to $K$ to result in the $(n+1)$-component unlink. Figure \ref{fig:bandmoves} shows what we will call an \emph{oriented rational band move}, so named because it corresponds to an oriented band move in a rational homology cobordism. Note that an ordinary oriented band move is a special case of an oriented rational band move.

\begin{figure}[htb!]
\begin{center}
\includegraphics[scale=0.75]{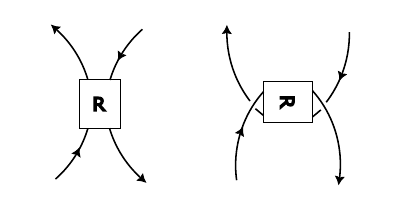}
\includegraphics[scale=0.75]{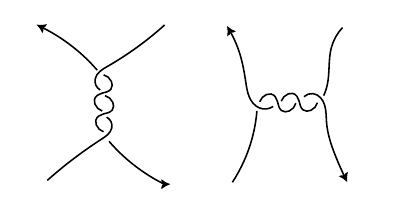}
\includegraphics[scale=0.75]{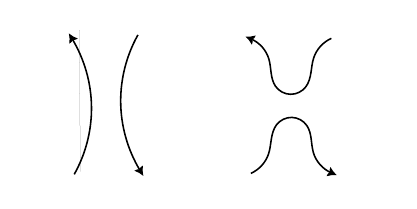}
\caption{Left, an oriented rational band move. The box represents any tangle with the given boundary orientations. Center, an example of a specific oriented rational band move. Right, an ordinary oriented band move, which is a special case of an oriented rational band move with the trivial tangle.}
\label{fig:bandmoves}
\end{center}
\end{figure}

We show that oriented rational band moves can be used to create rationally slice knots. Recall that a knot is called \emph{rationally slice} if it bounds a smooth properly embedded disk in a rational homology 4-ball.

\begin{theorem}\label{thm:Qbandmove}
Let $L$ be a $n$-component link in $S^3$. Suppose that $L'$ is a $(n+1)$-component link in $S^3$ obtained from $L$ by performing an oriented rational band move on the first component of $L$. Then there exist $n$ disjoint planar surfaces $F_1 \cup \dots \cup F_n$ smoothly properly embedded in a  rational homology $S^3 \times I$ such that
\begin{itemize}
\item  $F_i \cap \left(S^3 \times \{0\}\right)$ is the $i$th component of $L$, for each $i$,
\item $F_1 \cap \left(S^3 \times \{1\}\right)$ is the union of the first and the last component of $L'$,
\item   $F_i \cap \left(S^3 \times \{1\}\right)$ is the $i$th component of $L'$, for each $i>1$.
\end{itemize}

\noindent In particular, if a knot $K$ in $S^3$ can be made into the $(n+1)$-component unlink by $n$ oriented rational band moves, then $K$ is rationally slice.
\end{theorem}

See Figure \ref{fig:figureeight} for an explicit example.

\begin{figure}[htb!]
\begin{center}
\includegraphics[scale=0.9]{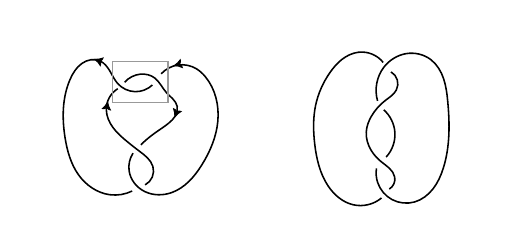}
\caption{An oriented rational band move from the figure eight knot to the unlink. The gray box indicates the tangle \textsf{R} in the oriented rational band move.}
\label{fig:figureeight}
\end{center}
\end{figure}

\begin{remark}
It follows from the proof of Theorem~\ref{thm:Qbandmove} that the 4-manifold in which $K$ bounds a smooth properly embedded disk is $(\#_n Z_0)^\circ$, where $Z_0$ is the closed rational homology 4-sphere from \cite[p.~2]{Levine} and $(\#_n Z_0)^\circ$ is $\#_n Z_0$ minus a small open ball. Moreover, the proof implies that, if $n$ oriented rational band moves can be performed disjointly, then $K$ is rationally slice in $Z_0^\circ$. See Section \ref{sec:Qbandmove} for more details.
\end{remark}


\subsection*{Acknowledgements}
JH is partially supported by NSF grant DMS-2104144 and a Simons Fellowship. Part of this work was done while JH was in residence at the Simons Laufer Mathematical Sciences Institute (formerly MSRI) during Fall 2023, supported by NSF Grant DMS-1928930.
MHK is partially supported by Samsung Science and Technology Foundation (SSTF-BA2202-01) and the NRF Young researcher program (2021R1C1C1012939). JP is partially supported by Samsung Science and Technology Foundation (SSTF-BA2102-02) and the POSCO TJ Park Science Fellowship. ZW is partially supported by a grant from the Hong Kong Research Grants Council (Project No. 14301819) and a direct grant from CUHK (Project No. 4053574). Part of this work was done during the Frontiers in Geometry and Topology Summer School and Research Conference at ICTP during Summer 2022 and during the Low-dimensional Topology Workshop at KIAS in Spring 2019.

\subsection*{Conventions}
For a given manifold $X$, we write $X^\circ$ to denote $X$ minus a small open ball. Given a oriented 3-manifold $Y$, let $-Y$ be the manifold $Y$ with reversed orientation. Given a knot $K$ in $Y$, let $-K \subset -Y$ be the knot obtained by reversing both the orientations of $K$ and $Y$. Throughout the article, we work in the smooth category.

\section{An oriented rational band move}\label{sec:Qbandmove}

In this section, we will prove Theorem \ref{thm:Qbandmove}. The proof is similar in spirit to \cite{FintushelStern}; see also \cite[Theorem 4.16]{ChaMAMS}. 

Let $Y$ and $Y'$ be rational homology 3-spheres. Recall that two $n$-component links $L \subset Y$ and $L' \subset Y'$ are \emph{rationally concordant}, if  there  exist $n$ disjoint annuli $A_1 \cup \dots \cup A_n$ smoothly properly embedded in a  rational homology $S^3 \times I$ such that $A_1 \cup \dots \cup A_n$  cobound $L$ and $L'$, and $A_i \cap \left(S^3 \times \{0\}\right)$ is the $i$th component of $L$. Lastly, let $\widetilde{U}$ denote the core of surgery in $S^3_2(\textup{unknot}) = L(2,1) $; that is, $\widetilde{U}$ is the ``unknotted'' generator of $H_1(L(2,1), \Z)$.

We begin with the following observation:

\begin{lemma}\label{lem:concinY}
Let $Y$ be a rational homology 3-sphere. If $K \subset Y$ is concordant to $J \subset Y$ in $Y \times I$, then $K \# -J \subset Y \# -Y$ bounds a disk in the rational homology 4-ball $Y^\circ \times I$. In particular, if $L$ is link in $S^3$, then $$L \subset S^3 \qquad\text{ and }\qquad L\#\widetilde{U}\#-\widetilde{U} \subset L(2,1) \# - L(2,1)$$ are rationally concordant in $\left(L(2,1)^\circ \times I\right)^\circ$. 
\end{lemma}


\begin{remark}
In order to be well-defined, the connected sum $L \# \widetilde{U}$ requires a choice of component of $L$.  Lemma~\ref{lem:concinY} is true for any choice of component. It will be clear in the proofs of Lemma~\ref{lem:Qband} and Theorem~\ref{thm:Qbandmove} which component is being used (the choice depends on where the oriented rational band move in performed), but this choice has no impact on the proofs.
\end{remark}


\begin{proof}
Suppose that $(Y, K)$ is concordant to $(Y, J)$ in $Y \times I$ via an annulus $A$. Remove a small tubular neighborhood of an arc $\gamma \subset A$ that goes from $K$ to $J$. The result is a disk in $Y^\circ \times I$ whose boundary is $K \# -J$ in $Y \# -Y$. In particular, we have that $\widetilde{U}\#-\widetilde{U} \subset L(2,1) \# - L(2,1)$ is rationally slice in $L(2,1)^\circ \times I$. This immediately implies the latter part of the statement.
\end{proof}

\begin{lemma}\label{lem:Qband}
Let $L$ be a $n$-component link in $S^3$ and $M$ be the rational homology 3-sphere $L(2,1) \# - L(2,1)$. Suppose that $L'$ is a $(n+1)$-component link in $S^3$ obtained from $L$ by an oriented rational band move performed on the first component of $L$.
Then there exist $n$ disjoint planar surfaces $F_1 \cup \dots \cup F_n$ smoothly properly embedded in $M \times I$ such that 
\begin{itemize}
\item  $F_i \cap \left(M \times \{0\}\right)$ is the $i$th component of $L\# \widetilde{U}\#-\widetilde{U} \subset M$, for each $i$, 
\item $F_1 \cap \left(M \times \{1\}\right)$ is the union of the first and the last component of $L'\# \widetilde{U}\#-\widetilde{U} \subset M$,
\item   $F_i \cap \left(M \times \{1\}\right)$ is the $i$th component of $L'\# \widetilde{U}\#-\widetilde{U} \subset M$, for each $i>1$. 
\end{itemize}\end{lemma}

\begin{proof}
The proof is shown in Figure~\ref{fig:bandmovesequence}. We start with the connected sum of $L \subset S^3$ with $\widetilde{U}\#-\widetilde{U} \subset M$, and perform an oriented band move as shown. Note that the band interacts nontrivially with $H_1(M; \Z)$. Next, we perform a handle slide, followed by a sequence of isotopies. The careful reader will observe that Figure~\ref{fig:bandmovesequence} concludes with $L' \# \widetilde{U}^r \#-\widetilde{U} \subset M$, where $\widetilde{U}^r$ denotes $\widetilde{U}$ with its string orientation reversed. But in $L(2,1)$, we have that $\widetilde{U}$ and $\widetilde{U}^r$ are isotopic.
\end{proof}

\begin{figure}[htb!]
\begin{center}
\labellist
	\pinlabel \textcolor{blue}{\tiny{$2$}} at 70 177
	\pinlabel \tiny{isotopy} at 87 157
	\pinlabel \tiny{band move} at 189 157
	\pinlabel \tiny{handle slide} at 282 157
	\pinlabel \tiny{isotopy} at 42 57
	\pinlabel \tiny{isotopy} at 130 57
	\pinlabel \tiny{isotopy} at 221 57
	\pinlabel \textcolor{blue}{\tiny{$2$}} at 234 27
\endlabellist
\includegraphics[scale=1]{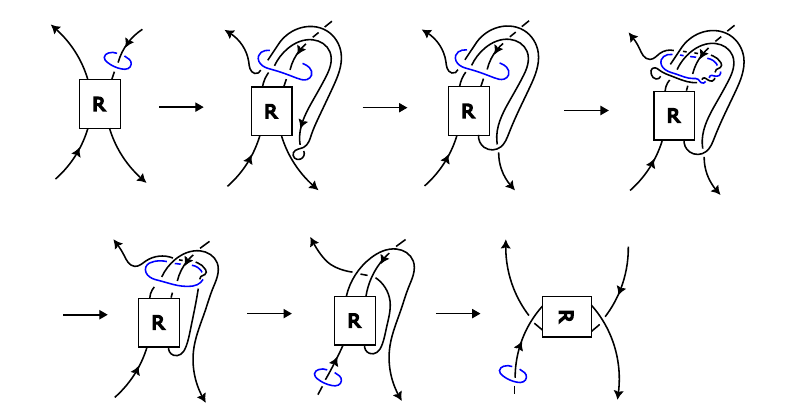}
\caption{An oriented band move in $M= L(2,1) \# - L(2,1)$, followed by a sequence of handle slides and isotopies. Here, we omit $-L(2,1)$ from the figure.}
\label{fig:bandmovesequence}
\end{center}
\end{figure}
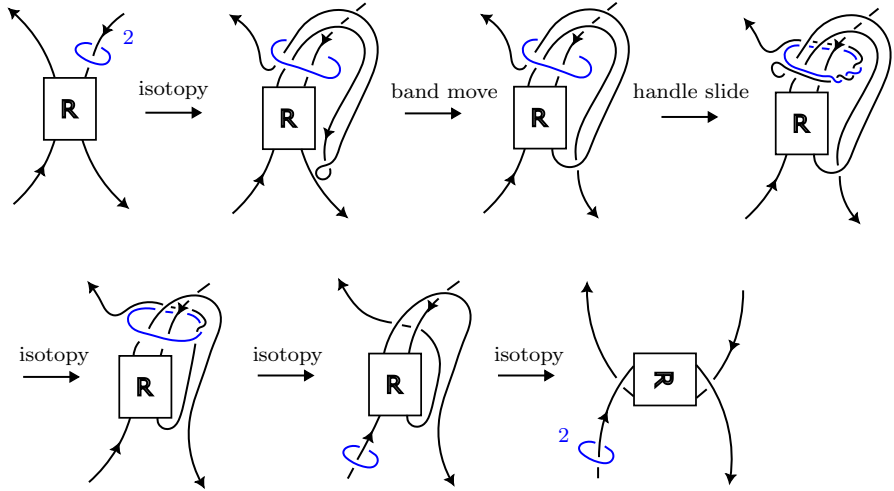

%

%

We are now ready to prove Theorem \ref{thm:Qbandmove}.

\begin{proof}[Proof of Theorem \ref{thm:Qbandmove}]
Suppose that $L'$ is a $(n+1)$-component link in $S^3$ obtained from $L$ by performing an oriented rational band move on the first component of $L$. Again, let $M$ be the rational homology 3-sphere $L(2,1) \# - L(2,1)$. Then by Lemma~\ref{lem:Qband}, there exist $n$ disjoint planar surfaces $F_1 \cup \dots \cup F_n$ smoothly properly embedded in $M \times I$ that cobound $$L\#\widetilde{U}\#-\widetilde{U} \subset M \qquad\text{ and }\qquad L'\#\widetilde{U}\#-\widetilde{U} \subset M.$$ 


\noindent Moreover, $F_1 \cup \dots \cup F_n$ satisfy the conclusion of Lemma~\ref{lem:Qband}. Then by concatenating concordances  $$\text{from }L \subset S^3 \text{ to } L\#\widetilde{U}\#-\widetilde{U} \subset M
\qquad \text{ and } \qquad\text{from }L'\#\widetilde{U}\#-\widetilde{U} \subset M\text{ to }L' \subset S^3$$ given  by  Lemma~\ref{lem:concinY} with $F_1 \cup \dots \cup F_n$, we obtain the desired smooth properly embedded surfaces in a rational homology $S^3\times I$. Note that by Lemma~\ref{lem:concinY}, the rational homology $S^3\times I$ is diffeomorphic to $$\left(L(2,1)^\circ \times I\right)^\circ \cup_M -\left(L(2,1)^\circ \times I\right)^\circ.$$

Finally, suppose that the $(n+1)$-component unlink in $S^3$ is obtained from a knot $K \subset S^3$ by $n$ oriented rational band moves. This implies that each oriented rational band move increases the number of the components by one. Then it is clear that by stacking smooth properly embedded planar surfaces from above $n$ times, we get a connected planar surface smoothly properly embedded in a rational homology $S^3 \times I$ that connects $K \subset S^3$ and the $(n+1)$-component unlink in $S^3$. Since the unlink bounds smooth properly embedded $(n+1)$ disjoint disks in $B^4$, by gluing this with the planar surface, we obtain the desired disk in a rational homology 4-ball.\end{proof}

\begin{remark}
In \cite[p.\ 2]{Levine}, he proves that every strongly negative amphichiral knot in $S^3$ bounds a smooth properly embedded disk in $Z_0^\circ$, where $Z_0$ is the double of the $D^2$-bundle over $\mathbb{R}\mathrm{P}^2$ with orientable total space and Euler number $0$. It can be easily verified that $Z_0$ can be also obtained as the double of $L(2,1)^\circ \times I$ (see also \cite[Remark~1.2]{Levine} and \cite[Figure 1]{Levine}). 

Now, suppose that a knot $K$ in $S^3$ can be made into the $(n+1)$-component unlink by $n$ oriented rational band moves. Then it follows from  the proof of Theorem~\ref{thm:Qbandmove} that the 4-manifold in which $K$ bounds a smooth properly embedded disk is $(\#_n Z_0)^\circ$. Moreover, if the $n$ oriented rational band moves are performed disjointly, we may use the same blue circle with framing 2 in Figure~\ref{fig:bandmovesequence}, to perform the  oriented band move $n$ times to conclude that $K$ bounds a smooth properly embedded disk in $Z_0^\circ$.\end{remark}

\bibliographystyle{amsalpha}
\bibliography{bib}

\end{document}